\documentclass[11pt,reqno]{amsart}
\usepackage{enumerate}
\usepackage{amsfonts, amsmath, amssymb, amscd, amsthm, bm}

\usepackage{enumerate}
\usepackage{url}
\usepackage[linktocpage=true,colorlinks,citecolor=magenta,linkcolor=blue,urlcolor=magenta]{hyperref}
\usepackage{multicol}
\usepackage{comment}
\allowdisplaybreaks[1] % Avoid too long aligns.

 \newtheorem{thm}{Theorem}[section]

 \newtheorem{lem}[thm]{Lemma}
 
 \newtheorem{prop}[thm]{Proposition}
 \theoremstyle{definition}
 \newtheorem{defn}{Definition}
 \newtheorem{rem}[thm]{Remark}

  \newtheorem{prob}{Problem}
 \newtheorem{condition}{Condition}
 \numberwithin{equation}{section}
\newcommand{\tr}{\mathrm{tr}\,}
\newcommand{\LF}{\mathcal{L}}
\newcommand{\df}{\dot{f}}
\newcommand{\ddf}{\ddot{f}}
\newcommand{\abs}[1]{\lvert #1\rvert}
\usepackage{xcolor}

\numberwithin{equation}{section}

\newcounter{rom}
\renewcommand{\therom}{(\roman{rom})}
{\end{list}}

\title[Self-similar solutions to fully nonlinear curvature flows]%
{Self-similar solutions to fully nonlinear curvature flows by high powers of curvature}
\begin{document}

\author[S. Gao]{Shanze Gao}
\address{School of Mathematics and Statistics,
	Shaanxi Normal University, Xi'an 710119, P. R. China}
\email{\href{mailto:gaoshanze@snnu.edu.cn}{gaoshanze@snnu.edu.cn}}

\author[H. Li]{Haizhong Li}
\address{Department of Mathematical Sciences,
Tsinghua University,  Beijing 100084, P. R. China}
\email{\href{mailto:lihz@tsinghua.edu.cn}{lihz@tsinghua.edu.cn}}

\author[X. Wang]{Xianfeng Wang}
\address{School of Mathematical Sciences and LPMC,
Nankai University,
Tianjin 300071,  P. R. China}
\email{\href{mailto:wangxianfeng@nankai.edu.cn}{wangxianfeng@nankai.edu.cn}}


\begin{abstract} In this paper, we investigate closed strictly convex hypersurfaces in $\mathbb{R}^{n+1}$ which shrink self-similarly under a large family of fully nonlinear curvature flows  by high powers of curvature. When the speed function is given by powers of a homogeneous of degree $1$ and inverse concave function of the principal curvatures with power greater than $1$, we prove that the only such hypersurfaces are round spheres. We also prove that slices are the only closed strictly convex self-similar solutions to such curvature flows in the hemisphere $\mathbb{S}^{n+1}_{+}$ with power greater than or equal to $1$.
\end{abstract}

\date{}

 \maketitle

%%%%%%%%%%%%%%%%%%%
\section{Introduction}
\label{Sec:Intro}
%%%%%%%%%%%%%%%%%%%

%Let $X:M^n\rightarrow \mathbb{R}^{n+1}$ be a closed hypersurface in  Euclidean space $\mathbb{R}^{n+1}$ with $n\geq 2$.
In the influential paper \cite{HUISKEN1984}, Huisken proved that
 for any strictly convex initial hypersurface $M^n$ in Euclidean space $\mathbb{R}^{n+1}$ with $n\geq 2$,
 there exists a unique smooth solution to the mean curvature flow and the solution contracts to a ``round'' point in finite time.
 We call a hypersurface given by $X:M^n\to\mathbb{R}^{n+1}$ a \textit{self-shrinker} if it satisfies the following elliptic equation
\begin{equation}
H=\langle X,\nu\rangle,
\end{equation}
where $H$ is the mean curvature and $\nu$ is the outward unit normal vector of $M^n$ in $\mathbb{R}^{n+1}$.
 It is well known that self-shrinkers play an important role in the study of the mean curvature flow (cf. \cite{c-m,h90}).
 Not only they correspond to self-shrinking solutions to the mean curvature flow, but also they describe all possible Type I blow-ups at a given singularity of the mean curvature flow.
 In another celebrated paper \cite{h90}, Huisken proved that
spheres are the only compact self-shrinkers in $\mathbb{R}^{n+1}$ with nonnegative mean curvature.

\subsection{Self-similar solutions to contracting curvature flows in Euclidean space}\label{sec1.1}
After the famous work of Huisken, there have been lots of beautiful results on contracting curvature
flows in Euclidean space. Let $M^n$ be a smooth, closed manifold with $n\geq 2$ and $X_0:M^n\to\mathbb{R}^{n+1}$ be a smooth immersion which is strictly convex,
we consider the contracting curvature flow given by a family of smooth immersions $X(\cdot ,t):M^n \times [0,T) \to \mathbb{R}^{n+1}$ solving the evolution equation
\begin{equation}\begin{aligned}\label{flow}
\left\{
\begin{array}{ll}
\frac{\partial X}{\partial t}(\cdot, t) = - F^\alpha (\cdot, t)\nu(\cdot, t),\\
X(\cdot , 0) = X_0 (\cdot),
\end{array}
\right.
\end{aligned}\end{equation}
where $\alpha\geq1$,  $\nu$ is the outward unit normal vector field of $M_t=X_t(M^n)$ and  $F$ is a homogeneous of degree $1$  function of  the principal  curvatures of $M_t$.
Analogous to the mean curvature flow case, we call $X:M^n\to\mathbb{R}^{n+1}$ a \textit{self-similar solution} to the flow \eqref{flow} if it
satisfies
\begin{equation}\label{1-1}
F^{\alpha}=\langle X, \nu \rangle,
\end{equation}
where $\nu$ is the outward unit normal vector field of $M^n$, $\alpha(\geq 1)$ is a constant and
$F(\mathcal{W})=f(\kappa(\mathcal{W}))$ is a smooth, symmetric function of the principal curvatures
$
(\kappa_1,\ldots,\kappa_n),
$
which are defined by eigenvalues of the Weingarten matrix $\mathcal{W}=(h_i^j)$ of $M^n$. 

Let $\Gamma_+=\{(\kappa_1,\ldots,\kappa_n)\in \mathbb{R}^n:~\kappa_i>0,\text{ for any }~i=1,2,...,n\}$ be the
positive quadrant in $\mathbb{R}^n$. We assume that $F(\mathcal{W})=f(\kappa(\mathcal{W}))$  satisfies the following properties.
\begin{condition}\label{condtn}
\begin{itemize}
  \item[(i)] $f$ is smooth, symmetric and positive on $\Gamma_+$.
  \item[(ii)] $f$ is strictly increasing in each argument, i.e., ${\partial f}/{\partial \kappa_i}>0$ on $\Gamma_+$, for any $i=1,\ldots,n$.
  \item[(iii)] $f$ is homogeneous of degree $1$, i.e., $f(k\kappa)=kf(\kappa)$ for any $k>0$ and $\kappa=(\kappa_1,\ldots,\kappa_n)\in \Gamma_+$.
  %\item[(iv)] $f$ is normalized such that $f(1,\ldots,1)=n$.
\end{itemize}
\end{condition}

It can be checked straightforwardly (cf. \cite{McCoy11}) that the homothetic immersions given by
\begin{equation}\label{homothetic}
X(\cdot, t)=((\alpha+1)(T-t))^{\frac{1}{\alpha+1}}X(\cdot)
\end{equation}
give rise to the solution of \eqref{flow} up to  tangential diffeomorphisms if $X$ satisfies \eqref{1-1}, and the corresponding
self-similar hypersurfaces shrink to a point at time $T$.
Thus self-similar solutions to curvature flows can be regarded as natural generalizations to self-shrinkers for mean curvature flow
and the study of them is crucial for the understanding of the asymptotic behavior of the corresponding curvature flows.
In the following, we recall some known results for contracting curvature flows and self-similar solutions in $\mathbb{R}^{n+1}$.

When $\alpha=1$, the contracting curvature flows have been studied intensively.
Chow studied the flows by the n-th root of the Gauss curvature \cite{CHOW1985} and the square root of the scalar curvature \cite{CHOW1987} (with an initial pinching condition). In a series of papers \cite{A1994CVPDE,BAndrews07,A2010CVPED}, by proving some powerful pinching estimates, Andrews  extended the results of Huisken and Chow to several wide classes of curvature flows, with speeds given by homogeneous of degree $1$ functions of the principal curvatures and satisfying some natural conditions.
McCoy \cite{McCoy11} obtained various classification results for a large class of fully nonlinear curvature flows with $\alpha=1$, which can be regarded as
analogues of Huisken's result for self-shrinkers. In particular, McCoy \cite{McCoy11} proved the following result.

\begin{thm}[\cite{McCoy11}]\label{Th:m}
Let $M^n$ be  a closed, strictly convex  hypersurface in $\mathbb{R}^{n+1}$ with $n\geq 2$, satisfying
\begin{equation*}\label{Eq:alpha=1}
F=\langle X, \nu\rangle,
\end{equation*}
where $F$ satisfies Condition \ref{condtn} and is inverse concave, then $M^n$ must be a round sphere.
\end{thm}

For the flow with $\alpha>1$, there are fewer results. The first prominent result was proved by  Andrews in \cite{An-99} for Gauss curvature flow, where  Firey's  conjecture that convex surfaces moving by their Gauss curvature become spherical as they contract to points was proved.
Recently, the generalized Firey's conjecture proposed by Andrews in \cite{Andrews1996} was completely solved \cite{g-n,AGN-16,BCD}, that is, the solutions of the flow by powers of the Gauss curvature  converge to spheres for any $\alpha>\frac{1}{n+2}$. This is a breakthrough in the study of curvature flows. In fact,  Andrews \cite{Andrews1996} proved that in the affine invariant case $\alpha=\frac{1}{n+2}$,
the flow converges to an ellipsoid.
Guan and Ni \cite{g-n} proved that convex hypersurfaces in $\mathbb{R}^{n+1}$ contracting by the Gauss curvature flow  converge (after rescaling to fixed volume) to a smooth uniformly convex self-similar solution of  the flow. Andrews, Guan and Ni \cite{AGN-16} extended the results in \cite{g-n}  to the flow by powers of the Gauss curvature $K^\alpha$ with $\alpha>\frac{1}{n+2}$.  Brendle, Choi and Daskalopoulos \cite{BCD} proved the following uniqueness result for self-similar solutions to the $K^{\alpha}$ flow with $\alpha> \frac{1}{n+2}$.

\begin{thm}[\cite{BCD}]\label{thm-bcd}
Let $M^n$ be  a closed, strictly convex  hypersurface in $\mathbb{R}^{n+1}$ with $n\geq 2$, satisfying
\begin{equation*}
K^\alpha=\langle X, \nu\rangle,
\end{equation*}
where $K$  is the Gauss curvature of $M^n$. If $\alpha> \frac{1}{n+2}$, then $M^n$ must be a round sphere.
\end{thm}

The asymptotical behavior for flow of convex hypersurfaces by arbitrary speeds (other than powers of Gauss curvature) which are smooth homogeneous functions of the principal curvatures of degree greater than $1$ is
still an open problem to be investigated. Under a certain curvature pinching condition, Andrews and McCoy \cite{AM-12} proved that the flows converge to round spheres in finite time after proper rescaling. Andrews, McCoy and Zheng \cite{A-M-Z} constructed some examples of singular behavior in flow of convex hypersurfaces with arbitrary power $\alpha$, which show that initially smooth and uniformly convex hypersurfaces can evolve to become non-smooth or non-convex.  For more results about contracting curvature flows with power $\alpha>1$, we refer to the recent paper by the second and the third authors together with Wu \cite{LWW} and the references therein. Andrews, McCoy and Zheng's result \cite{A-M-Z} indicates that it is a challenging
problem to capture the whole picture of  flow of convex hypersurfaces by arbitrary speeds.

Based on the resolution of the the generalized Firey's conjecture  mentioned
above, one can expect that the study of self-similar solutions will be important for understanding the asymptotical behavior of flow of convex hypersurfaces by arbitrary speeds.
In this respect, there are some important progresses recently.
In order to state some of the existing results and the main result of this paper, we recall an additional condition of the curvature function $F(\mathcal{W})=f(\kappa(\mathcal{W}))$ and the definition of inverse concave function.
\begin{condition}\label{condtn2}
For all $(y_1,\ldots,y_n)\in\mathbb{R}^n$, we have
\begin{equation}
\sum_i\frac{1}{\kappa_i}\frac{\partial \log f}{\partial \kappa_i}y_i^2+\sum_{i,j}\frac{\partial^2\log f}{\partial \kappa_i\partial \kappa_j}y_i y_j\geq 0.
\end{equation}
\end{condition}

\begin{defn}
  We say that $f$ is inverse concave if the function
  \begin{equation}\label{s1:f*}
    f_*(x_1,\ldots,x_n):=\frac 1{f(\frac 1{x_1},\ldots,\frac 1{x_n})}
  \end{equation}
  is concave.
\end{defn}
\begin{rem}
It is checked in Lemma \ref{lem:condn2} that if $f$ satisfies Condition \ref{condtn} and  Condition \ref{condtn2}, then $f$ is inverse concave. Therefore, for functions satisfying Condition \ref{condtn}, inverse concavity is a weaker condition than Condition \ref{condtn2}.
For more properties of inverse concave functions and some examples, we refer to Section \ref{sec2}.
\end{rem}

By  adapting the test functions introduced by Choi-Daskalopoulos \cite{c-d} and Brendle-Choi-Daskalopoulos \cite{BCD}
and exploring the properties of the $k$-th elementary symmetric function $\sigma_k$ intensively, the first two authors together with Ma \cite{GLM} proved the following uniqueness result for self-similar solutions to contracting curvature flows.

\begin{thm}[\cite{GLM}]\label{glmthm} Let $M^n$ be a closed strictly convex hypersurface in $\mathbb{R}^{n+1}$ with $n\geq 2$, satisfying
\begin{equation*}
F^{\alpha}=\langle X,\nu\rangle,
\end{equation*}
where $F$ satisfies Condition \ref{condtn} and Condition \ref{condtn2}. If $\alpha> 1$, then $M^n$ must be a round sphere.
\end{thm}

\begin{rem}Condition \ref{condtn2} and \eqref{eq:fkkappak} (which is actually implied by Condition \ref{condtn2}, see Lemma \ref{lem:condn2}) are two essential inequalities used in the proof of Theorem \ref{glmthm}.
We note that $F$ satisfies Condition \ref{condtn2} if and only if $F^{\alpha}(\alpha> 0)$ satisfies Condition \ref{condtn2}.
\end{rem}

Recently, by modifying the test functions introduced by Choi-Daskalopoulos \cite{c-d} and Brendle-Choi-Daskalopoulos \cite{BCD},
Chen  \cite{ChenJFA} extended the result in Theorem \ref{glmthm} to the case of $F=(\sigma_{n}/\sigma_{k})^{\frac{1}{n-k}}$, which is closely related to the $L_p$-Christoffel-Minkowski problem. Later, Chen and the first author \cite{CG19} extended the result to the case of $F=(\sigma_{k}/\sigma_{l})^{\frac{1}{k-l}}$ with $0\leq l<k\leq n$.

We notice that the functions $(\sigma_{k}/\sigma_{l})^{\frac{1}{k-l}}$ with $0\leq l<k\leq n$ are inverse concave, although
they do not satisfy Condition \ref{condtn2}. It is natural to propose the following problem.
\begin{prob}\label{prob1}
To classify the strictly convex self-similar solutions \eqref{1-1} to contracting curvature flows in $\mathbb{R}^{n+1}$ with $F$ satisfying Condition \ref{condtn} and being inverse concave.
\end{prob}

The aim of this paper is to solve the above problem. In fact, we solve Problem \ref{prob1} by proving the following uniqueness result.

\begin{thm}\label{main} Let $M^n$ be a closed strictly convex hypersurface in $\mathbb{R}^{n+1}$ with $n\geq 2$, satisfying
\begin{equation}\label{maine}
F^{\alpha}=\langle X,\nu\rangle,
\end{equation}
where $F$ satisfies Condition \ref{condtn} and is inverse concave. If $\alpha> 1$, then $M^n$ must be a round sphere.
\end{thm}

\begin{rem}
(i) The assumption of $F$ in Theorem \ref{main} is weaker than that in Theorem \ref{glmthm}. The result in Theorem \ref{main} covers previous results in \cite{ChenJFA},\cite{CG19} and can also be regarded as a generalization of Theorem \ref{Th:m} and Theorem \ref{glmthm}.
(ii) Similar to the proof of Theorem 1.12 in  \cite{GLM}, we can consider a slightly more general equation
\begin{equation}\label{neweq}
F^{\alpha}+C=\langle X, \nu\rangle
\end{equation}
where $C\leq 0$ is a constant. It can be proved by an argument analogous to that of Theorem \ref{main} that $M^n$ must be a round sphere if we replace the self-similar equation
\eqref{maine} by \eqref{neweq}.
\end{rem}

\subsection{Self-similar solutions to contracting curvature flows in  the hemisphere}\label{sec1.2}
Self-similar equation can be extended to more general spaces like warped product manifolds, see \cite{GaoMa-19} and its references. In \cite{GaoMa-19}, the uniqueness of self-similar solutions in the hemisphere were obtained when the speed function of the corresponding curvature flow  satisfies Condition 1.8 in \cite{GLM}. In general, self-similar solutions to curvature flows in warped product manifolds (other than Euclidean space) do not arise from blow-up procedures, however, they provide barriers and the knowledge of them will be important in
the study of the asymptotical behavior of the corresponding curvature flows.

Let $N=[0,\bar{r}	)\times \mathbb{S}^{n}$ be a warped product manifold with warped product metric 
$$\bar{g}=dr\otimes dr+\lambda^{2}(r)\sigma,$$
where $\sigma$ denotes the standard metric of $\mathbb{S}^{n}$.
 In the sequel, we consider the case that $N$ has constant sectional curvature $\epsilon$, that is, when $\lambda=r$ and $\bar{r}	=\infty$, $N$ is Euclidean space $\mathbb{R}^{n+1}$ with constant sectional curvature $\epsilon=0$; when $\lambda=\sin r$ and $\bar{r}	=\frac{\pi}{2}$, $N$ is the hemisphere $\mathbb{S}^{n+1}_{+}$ with constant sectional curvature $\epsilon=1$; when $\lambda=\sinh r$ and $\bar{r}	=\infty$, $N$ is  hyperbolic space $\mathbb{H}^{n+1}$ with constant sectional curvature $\epsilon=-1$. Let $M^n$ be a closed strictly convex hypersurface in $N$, the corresponding self-similar equation is defined by
\begin{equation}\label{Eq:ssN2}
	F^{\alpha}=\bar{g}(\lambda\partial_{r},\nu),
\end{equation}
where $\nu$ is the outward unit normal vector field of $M^n$, $\alpha(\geq 1)$ is a constant  and
$F(\mathcal{W})=f(\kappa(\mathcal{W}))$ is a smooth, symmetric function of the principal curvatures
$
(\kappa_1,\ldots,\kappa_n)
$
of $M^n$.

We prove analogues of Theorem \ref{Th:m} and Theorem \ref{main} in the hemisphere $\mathbb{S}^{n+1}_{+}$.

\begin{thm}\label{Th:hemisphere}
 Let $M^n$ be a closed strictly convex hypersurface in  $\mathbb{S}^{n+1}_{+}$ with $n\geq 2$, satisfying \eqref{Eq:ssN2}.
If $\alpha\geq 1$,  $F$ satisfies Condition \ref{condtn} and is inverse concave, then $M^n$ is a slice $\{ r_{0} \}\times \mathbb{S}^{n}$ in $\mathbb{S}^{n+1}_{+}$.
\end{thm}

\subsection{Outline of the proofs and organization of the paper}\label{sec1.3}
In Section \ref{sec2}, we give some notations and preliminary results. We give two different proofs of Theorem \ref{main} and Theorem \ref{Th:hemisphere}.
Section \ref{Sec:Z} is devoted to the first proof of Theorem \ref{main}. We introduce an auxiliary quantity
\begin{equation*}
Z=F^{\alpha}\cdot\frac{\abs{b}^{2}}{\tr b}-\frac{\alpha-1}{2\alpha}\abs{X}^{2},
\end{equation*}
where $b$ is the inverse of the second fundamental form $h$. The advantage of $Z$ is that we can apply strong maximum principle to $Z$ near a umbilical point of $M^n$ for $F$ with only inverse concavity. In Section \ref{sec3.1}, we obtain some equations involving the quantity $Z$ and establish some crucial estimates. To use the strategy of Brendle-Choi-Daskaspoulos \cite{BCD} (also see \cite{c-d}), we  need to analyze $W=\frac{F^{\alpha}}{\kappa_{\min}}-\frac{\alpha-1}{2\alpha}\abs{X}^{2}$ at its maximum point where $\kappa_{\min}$ is the smallest principal curvature. 
Conditions of $F$ in Theorem \ref{main} are enough for us to accomplish this part, which is inspired by \cite{CG19}. Thus, we can prove Theorem \ref{main}.
In Section \ref{Sec:hemisphere}, we give the proof of Theorem \ref{Th:hemisphere}  by using some modifications of the proof of Theorem \ref{main}.

In Section \ref{sec:proof2}, we give an alternative proof of Theorem \ref{main} and Theorem \ref{Th:hemisphere}.
This is achieved by  applying the  maximum principle to the tensor
\begin{equation*}
	T_{kl}= F^\alpha  b^{kl}-\frac{\alpha-1}{\alpha}\Phi g^{kl}-\beta g^{kl},
\end{equation*}
where $(b^{kl})$  is the inverse matrix of $(h_{kl})$, $\Phi(r)=\int_{0}^{r}\lambda(s)ds$, $g^{kl}$ is the inverse of the metric and $\beta$ is a constant.
This kind of maximum principle was used by McCoy in \cite{McCoy11} (cf. \cite{BAndrews07,A-M-Z}) to prove Theorem \ref{Th:m}.

\textbf{Acknowledgments:}
The authors would like to thank Professor James McCoy for
his interest and valuable comments.
H. Li  was supported by NSFC Grant  No.11831005 and NSFC-FWO 11961131001. X. Wang was supported by NSFC Grant No.11971244, Natural Science Foundation of Tianjin, China (Grant No.19JCQNJC14300) and the Fundamental Research Funds for the Central Universities, and she would also like to express her deep gratitude to the Mathematical Sciences Institute at the Australian National University for its hospitality and to Professor  Ben Andrews for his encouragement and help during her stay in MSI of ANU as a Visiting Fellow, while part of this work was completed. The authors would also like to thank the referee  for the  valuable comments and suggestions.

%%%%%%%%%%%%%%%%%%%%%%%%%
\section{Notations and preliminaries}
\label{sec2}
%%%%%%%%%%%%%%%%%%%%%%%%%

Throughout this paper, repeated indices will be summed unless otherwise stated.

Let $N=[0,\bar{r}	)\times \mathbb{S}^{n}$ be a warped product manifold with metric $\bar{g}=dr\otimes dr+\lambda^{2}(r)\sigma$ which has constant sectional curvature $\epsilon$. Let $M^n$ be a closed strictly convex hypersurface in $N$. Recall that the self-similar equation is given by 
\begin{equation}\label{Eq:ssN}
F^{\alpha}=\bar{g}(\lambda\partial_{r},\nu).
\end{equation}

Suppose that $\{ e_{1}, e_{2},\ldots,e_{n} \}$ is an orthonormal frame on $M^n$. 
Let  $
(\kappa_1,\ldots,\kappa_n)
$ be the principal curvatures
of $M^n$ which are the eigenvalues of the second fundamental form  $h=(h_{ij})$ on $M^n$.  We use $\nabla$ to denote the Levi-Civita connection on $M^n$. For convenience, we denote $\nabla_{k}h_{ij}=h_{ijk}$. 
We use $\LF$ to denote the operator $\LF=\alpha F^{\alpha-1}\frac{\partial F}{\partial h_{ij}}\nabla_{i}\nabla_{j}$. Under Condition \ref{condtn}, $F>0$ and $(\frac{\partial F}{\partial h_{ij}})$ is positive definite, thus $\LF$ is an elliptic operator. We will use the operator $\LF$ to establish some basic equations in Section \ref{Sec:Z}.
Denote $\Phi(r)=\int_{0}^{r}\lambda(s)ds$. We remark that $\Phi=\frac{\abs{X}^{2}}{2}$ in Euclidean space $\mathbb{R}^{n+1}$, where $X$ is the position vector.

In the following, we recall some basic properties of symmetric functions on $M^n$.
%\subsection{Symmetric functions}\label{sec2} $\ $

Given a smooth symmetric and positive function $f$ on the positive cone $\Gamma_+\subset \mathbb{R}^n$, a result of Glaeser \cite{Gla} implies that there is a smooth $GL(n)$-invariant function $F$ on the space $\mathrm{Sym}(n)$ of symmetric matrices such that $f(\kappa(A))=F(A)$, where $\kappa(A)=(\kappa_1,\ldots,\kappa_n)$ are the eigenvalues of $A$. We denote by $\dot{F}^{ij}$ and $\ddot{F}^{ij,kl}$ the first and second derivatives of $F$ with respect to the components of its argument.  We also use the notations $\dot{f}^i(\kappa)$, $\ddot{f}^{ij}(\kappa)$ to denote the derivatives of $f$ with respect to $\kappa$. At any diagonal $A$, we have
\begin{equation*}
  \dot{F}^{ij}(A)=\dot{f}^i(\kappa(A))\delta_i^j.
\end{equation*}
If the eigenvalues of $A$ are mutually different, the second derivative $\ddot{F}$ of $F$ in direction $B\in \mathrm{Sym}(n)$ is given in terms of $\dot{f}$ and $\ddot{f}$ by  (see e.g., \cite{BAndrews07}):
\begin{equation}\label{F-dd}
  \ddot{F}^{ij,kl}(A)B_{ij}B_{kl}=\sum_{i,k}\ddot{f}^{ik}(\kappa(A))B_{ii}B_{kk}+2\sum_{i>k}\frac{\dot{f}^i(\kappa(A))-\dot{f}^k(\kappa(A))}{\kappa_i(A)-\kappa_k(A)}B_{ik}^2.
\end{equation}
This formula makes sense as a limit in the case of  $\kappa_i=\kappa_k$. We have the following properties for concave functions. See e.g., \cite{BAndrews07,A-M-Z} for the proof.
\begin{lem}\label{lem-conc}
A smooth symmetric function $F$ on $\mathrm{Sym}(n)$ is concave in $A$ if and only if $f$ is concave in $\kappa(A)$ and $(\dot{f}^i-\dot{f}^k)(\kappa_i-\kappa_k)\leq 0$ for any $i\neq k$. If $f$ satisfies Condition \ref{condtn} and is concave, then
\begin{equation}\label{conc}
  \sum_{i}\dot{f}^{i}\geq f(1,\ldots,1),\qquad \ f\leq \frac{f(1,\ldots,1)}n\sum_i\kappa_i.
\end{equation}
\end{lem}
The examples of concave symmetric functions include: (i) $E_k^{1/k}$; (ii)  $(E_k/{E_l})^{1/{(k-l)}}$ with $k>l$, where
\begin{equation*}
  E_k=\binom{n}{k}^{-1}\sigma_k(\kappa)=\binom{n}{k}^{-1}\sum_{1\leq i_1<\cdots<i_k\leq n}\kappa_{i_1}\cdots\kappa_{i_k};
\end{equation*}
and (iii) the power means $  H_r=(\sum_{i=1}^n\kappa_i^r)^{1/r} $ with $r\leq 1$. Taking convex combinations or geometric means of the above concave examples can produce more concave examples.

For any positive definite symmetric matrix $A\in \mathrm{Sym}(n)$, we define $F_*(A)=F(A^{-1})^{-1}$. Then $F_*(A)=f_*(\kappa(A))$, where $f_*$ is the dual function of $f$ defined in \eqref{s1:f*}.
% \begin{equation}\label{s2:f-dual}
%      f_*(x_1,\ldots,x_n)=f(x_1^{-1},\ldots,x_n^{-1})^{-1}.
%    \end{equation}
We say that a symmetric function $F$ is inverse-concave if $F_*(A)$ is concave. The following lemma characterizes the inverse concavity of $f$ and $F$ (see \cite{BAndrews07,A-M-Z,A-W}).
\begin{lem}\label{inv-concave}
\begin{itemize}
  \item[(i)] $F_*$ is concave on $\Gamma_+$ if and only if $f_*$ is concave on $\Gamma_+$.
   \item[(ii)] $f$ is inverse concave if and only if
 \begin{equation}\label{inv-conc-1}
   \sum_{k,l=1}^n\ddot{f}^{kl}y_ky_l+2\sum_{k=1}^n\frac {\dot{f}^k}{\kappa_k}y_k^2~\geq ~2f^{-1}(\sum_{k=1}^n\dot{f}^ky_k)^2
 \end{equation}
 for any $y=(y_1,\ldots,y_n)\in \mathbb{R}^n$. Moreover, if $f$ is inverse concave, then
\begin{equation}\label{s2:inv-conc}
\frac{\dot{f}^k-\dot{f}^l}{\kappa_k-\kappa_l}+\frac{\dot{f}^k}{\kappa_l}+\frac{\dot{f}^l}{\kappa_k}\geq~0, ~\left(\dot{f}^k\kappa_k^2-\dot{f}^l \kappa_{l}^2\right)(\kappa_k-\kappa_l)\geq 0,\quad \forall~k\neq l.
\end{equation}
 \item[(iii)] If $f$ satisfies Condition \ref{condtn}  and is inverse concave, then
\begin{equation}\label{s2:ic-2}
	\sum_{k=1}^n \dot{f}^k\kappa_k^2\geq~ f^2/f(1,\ldots,1).
\end{equation}
\end{itemize}
\end{lem}

The examples (i) $E_k^{1/k}$; (ii)  $(E_k/{E_l})^{1/{(k-l)}}$ with $k>l$; and (iii) the power means $H_r$ with $r\geq -1$ are smooth inverse-concave symmetric functions. Also, taking convex combinations or geometric means of the inverse-concave examples can produce more inverse-concave examples.

The following lemma gives some properties of $f$ which satisfies Condition \ref{condtn2}.
\begin{lem}\label{lem:condn2}
\begin{itemize}
\item[(i)] If $f$ satisfies Condition \ref{condtn2}, then, for any $k\neq l$,
\begin{equation}\label{eq:fkkappak}
\left(\dot{f}^k\kappa_{k}-\dot{f}^l \kappa_{l}\right)(\kappa_k-\kappa_l)\geq 0.
\end{equation}
\item[(ii)] If $f$ satisfies Condition \ref{condtn} and Condition \ref{condtn2}, then $f$ is inverse concave.
\end{itemize}
\end{lem}

\begin{proof}
(i) By direct computation, we know Condition \ref{condtn2} is equivalent to the convexity of $\tilde{f}(x)=\log f(e^{x})$, where $x\in \mathbb{R}^{n}$ and $e^{x}=(e^{x_{1}},...,e^{x_{n}})$. Then we obtain \eqref{eq:fkkappak} by applying Lemma 2.20 in \cite{A1994CVPDE}.\\
(ii) If $f$ satisfies Condition \ref{condtn2}, we know
\begin{equation*}
\sum_{k,l=1}^n\ddot{f}^{kl}y_ky_l+\sum_{k=1}^n\frac {\dot{f}^k}{\kappa_k}y_k^2~\geq ~f^{-1}(\sum_{k=1}^n\dot{f}^ky_k)^2
\end{equation*}
for any $y=(y_1,\ldots,y_n)\in \mathbb{R}^n$. By Cauchy-Schwarz inequality and homogeneity of $f$, we know
\begin{equation*}
\sum_{k=1}^n\frac {\dot{f}^k}{\kappa_k}y_k^2~\geq ~f^{-1}(\sum_{k=1}^n\dot{f}^ky_k)^2.
\end{equation*}
Combining these equations together, we obtain that $f$ satisfies \eqref{inv-conc-1}. This completes the proof.
\end{proof}

%%%%%%%%%%%%%%%%%%%%%%%%%%%%%%%%%%%%%%%%%%%%%%%%%%%%%
\section{Proof of Theorem \ref{main}}\label{Sec:Z}
In this section, we give the first proof of  Theorem \ref{main}.
Let $(b^{ij})$  denote the inverse matrix of $(h_{ij})$ and $G(b)$ be a homogeneous of degree $1$ function of $(b^{ij})$, we have the following basic equations.
\begin{prop}\label{Prop:formula}
For any $M^n$ satisfying \eqref{Eq:ssN}, we have the following equations.
\begin{enumerate}[(1)]
\item \begin{align*}
\LF F^{\alpha}
=\bar{g}( \lambda\partial_{r}, \nabla F^{\alpha})+\alpha\lambda' F^{\alpha}-\alpha F^{2\alpha-1}\frac{\partial F}{\partial h_{ij}}h_{jl}h_{li},
\end{align*}
\item \begin{align*}
\LF h_{kl}&=\bar{g}(\lambda\partial_{r},\nabla h_{kl})+\lambda'h_{kl}-\alpha(\alpha-1)F^{\alpha-2}\nabla_{k}F\nabla_{l}F-\alpha F^{\alpha-1}\frac{\partial^{2} F}{\partial h_{ij}\partial h_{st}}h_{ijk}h_{stl}\\
&\quad -\alpha F^{\alpha-1}\frac{\partial F}{\partial h_{ij}}h_{mj}h_{mi}h_{kl}+(\alpha-1) F^{\alpha}h_{km}h_{ml} \\
&\quad +\epsilon \alpha F^{\alpha-1} \frac{\partial F}{\partial h_{ij}}\Big(h_{il}\delta_{kj}-h_{ij}\delta_{kl}+h_{kl}\delta_{ij}-h_{kj}\delta_{il}\Big),
\end{align*}
\item \begin{align*}
\LF b^{kl}
&=\bar{g}(\lambda\partial_{r},\nabla b^{kl})-\lambda'b^{kl}+\alpha(\alpha-1)F^{\alpha-2}b^{kp}b^{ql}\nabla_{p}F\nabla_{q}F\\
&\quad +\alpha F^{\alpha-1}b^{kp}b^{ql}\frac{\partial^{2} F}{\partial h_{ij}\partial h_{st}}h_{ijp}h_{stq}+\alpha F^{\alpha-1}b^{kl}\frac{\partial F}{\partial h_{ij}}h_{mj}h_{mi}\\
&\quad -(\alpha-1) F^{\alpha}\delta_{kl}+2\alpha F^{\alpha-1}b^{ks}b^{pt}b^{lq}\frac{\partial F}{\partial h_{ij}}h_{sti}h_{pqj} \\
&\quad -\epsilon \alpha F^{\alpha-1} b^{kp}b^{ql} \frac{\partial F}{\partial h_{ij}}\Big(h_{iq}\delta_{pj}-h_{ij}\delta_{pq}+h_{pq}\delta_{ij}-h_{pj}\delta_{iq}\Big),
\end{align*}
\item \begin{align*}
\LF G&=\alpha F^{\alpha-1}\frac{\partial F}{\partial h_{ij}}\frac{\partial^{2} G}{\partial b^{kl}\partial b^{pq}}\nabla_{i}b^{pq}\nabla_{j}b^{kl}+\bar{g}(\lambda\partial_{r},\nabla G)-\lambda'G\\
&\quad +\alpha(\alpha-1)F^{\alpha-2}\frac{\partial G}{\partial b^{kl}}b^{kp}b^{ql}\nabla_{p}F\nabla_{q}F+\alpha F^{\alpha-1}\frac{\partial G}{\partial b^{kl}}b^{kp}b^{ql}\frac{\partial^{2} F}{\partial h_{ij}\partial h_{st}}h_{ijp}h_{stq}\\
&\quad +\alpha F^{\alpha-1}G\frac{\partial F}{\partial h_{ij}}h_{mj}h_{mi}-(\alpha-1) F^{\alpha}\sum_{i}\frac{\partial G}{\partial b^{ii}}+2\alpha F^{\alpha-1}\frac{\partial G}{\partial b^{kl}}b^{ks}b^{pt}b^{lq}\frac{\partial F}{\partial h_{ij}}h_{sti}h_{pqj}\\
&\quad +\epsilon \alpha F^{\alpha-1}\left( F\frac{\partial G}{\partial b^{kl}}b^{kp}b^{pl}-G\sum_{i}\frac{\partial F}{\partial h_{ii}} \right),
\end{align*}
\item \begin{align*}
\LF (F^{\alpha}G)&=2\alpha F^{\alpha-1}\frac{\partial F}{\partial h_{ij}}\nabla_{i}F^{\alpha}\nabla_{j}G+\bar{g}( \lambda\partial_{r},\nabla (F^{\alpha}G) )+\lambda'(\alpha-1) F^{\alpha}G\\
&\quad +\alpha F^{2\alpha-1}\frac{\partial F}{\partial h_{ij}}\frac{\partial^{2} G}{\partial b^{kl}\partial b^{pq}}\nabla_{i}b^{pq}\nabla_{j}b^{kl}+\alpha(\alpha-1)F^{2\alpha-2}\frac{\partial G}{\partial b^{kl}}b^{kp}b^{ql}\nabla_{p}F\nabla_{q}F\\
&\quad +\alpha F^{2\alpha-1}\frac{\partial G}{\partial b^{kl}}b^{kp}b^{ql}\frac{\partial^{2} F}{\partial h_{ij}\partial h_{st}}h_{ijp}h_{stq}-(\alpha-1) F^{2\alpha}\sum_{i}\frac{\partial G}{\partial b^{ii}}\\
&\quad +2\alpha F^{2\alpha-1}\frac{\partial G}{\partial b^{kl}}b^{ks}b^{pt}b^{lq}\frac{\partial F}{\partial h_{ij}}h_{sti}h_{pqj}+\epsilon \alpha F^{2\alpha-1}\left( F\frac{\partial G}{\partial b^{kl}}b^{kp}b^{pl}-G\sum_{i}\frac{\partial F}{\partial h_{ii}} \right),
\end{align*}
\item \begin{align*}
\LF \Phi=\lambda'\alpha F^{\alpha-1}\sum_{i}\frac{\partial F}{\partial h_{ii}}-\alpha F^{2\alpha}.
\end{align*}
\end{enumerate}
\end{prop}

\begin{proof}
Proofs of (1), (2), (3) and (6) are analogous to that in \cite{GLM} and \cite{GaoMa-19}. We only prove (4) and (5) here.

From
\begin{align*}
\nabla_{j}G&=\frac{\partial G}{\partial b^{kl}}\nabla_{j}b^{kl},\\
\nabla_{i}\nabla_{j}G&=\frac{\partial^{2} G}{\partial b^{kl}\partial b^{pq}}\nabla_{i}b^{pq}\nabla_{j}b^{kl}+\frac{\partial G}{\partial b^{kl}}\nabla_{i}\nabla_{j}b^{kl},
\end{align*}
we have
\begin{equation*}
\LF G=\alpha F^{\alpha-1}\frac{\partial F}{\partial h_{ij}}\frac{\partial^{2} G}{\partial b^{kl}\partial b^{pq}}\nabla_{i}b^{pq}\nabla_{j}b^{kl}+\frac{\partial G}{\partial b^{kl}}\LF b^{kl}.
\end{equation*}
Then we get by use of (3)
\begin{align*}
\LF G&=\alpha F^{\alpha-1}\frac{\partial F}{\partial h_{ij}}\frac{\partial^{2} G}{\partial b^{kl}\partial b^{pq}}\nabla_{i}b^{pq}\nabla_{j}b^{kl}+\bar{g}(\lambda\partial_{r},\nabla G)-\lambda'G \\
&\quad +\alpha(\alpha-1)F^{\alpha-2}\frac{\partial G}{\partial b^{kl}}b^{kp}b^{ql}\nabla_{p}F\nabla_{q}F+\alpha F^{\alpha-1}\frac{\partial G}{\partial b^{kl}}b^{kp}b^{ql}\frac{\partial^{2} F}{\partial h_{ij}\partial h_{st}}h_{ijp}h_{stq}\\
&\quad +\alpha F^{\alpha-1}G\frac{\partial F}{\partial h_{ij}}h_{mj}h_{mi}-(\alpha-1) F^{\alpha}\sum_{i}\frac{\partial G}{\partial b^{ii}}\\
&\quad +2\alpha F^{\alpha-1}\frac{\partial G}{\partial b^{kl}}b^{ks}b^{pt}b^{lq}\frac{\partial F}{\partial h_{ij}}h_{sti}h_{pqj}+\epsilon \alpha F^{\alpha-1}\left( F\frac{\partial G}{\partial b^{kl}}b^{kp}b^{pl}-G\sum_{i}\frac{\partial F}{\partial h_{ii}} \right),
\end{align*}
which is (4).

 Combining (1) and (4), we obtain (5).
\end{proof}

\subsection{Estimate of $\LF Z$}\label{sec3.1}

%%%%%%%%%%%%%%%%%%%%%%%%%%%%%%%%%%%%%%%%%%%%%%%%%
We consider an auxiliary quantity $Z=F^{\alpha}G-\frac{\alpha-1}{\alpha}\Phi$, where a concrete $G$ will be chosen later. Using (5) and (6) in Proposition \ref{Prop:formula}, we have
\begin{align*}
\LF Z&=\lambda'(\alpha-1) F^{\alpha}G-\lambda'(\alpha-1) F^{\alpha-1}\sum_{i}\frac{\partial F}{\partial h_{ii}}-(\alpha-1) F^{2\alpha}\sum_{i}\frac{\partial G}{\partial b^{ii}}+(\alpha-1) F^{2\alpha}\\
&\quad +\epsilon \alpha F^{2\alpha-1}\left( F\frac{\partial G}{\partial b^{kl}}b^{kp}b^{pl}-G\sum_{i}\frac{\partial F}{\partial h_{ii}} \right)+2\alpha F^{\alpha-1}\frac{\partial F}{\partial h_{ij}}\nabla_{i}F^{\alpha}\nabla_{j}G\\
&\quad +\bar{g}(\lambda\partial_{r},\nabla (F^{\alpha}G))+\alpha F^{2\alpha-1}\frac{\partial F}{\partial h_{ij}}\frac{\partial^{2} G}{\partial b^{kl}\partial b^{pq}}\nabla_{i}b^{pq}\nabla_{j}b^{kl}\\
&\quad +\alpha(\alpha-1)F^{2\alpha-2}\frac{\partial G}{\partial b^{kl}}b^{kp}b^{ql}\nabla_{p}F\nabla_{q}F\\
&\quad +\alpha F^{2\alpha-1}\frac{\partial G}{\partial b^{kl}}b^{kp}b^{ql}\frac{\partial^{2} F}{\partial h_{ij}\partial h_{st}}h_{ijp}h_{stq}\\
&\quad +2\alpha F^{2\alpha-1}\frac{\partial G}{\partial b^{kl}}b^{ks}b^{pt}b^{lq}\frac{\partial F}{\partial h_{ij}}h_{sti}h_{pqj}.
\end{align*}

Now we analyze the terms on the right-hand side by categorizing them to three parts $L_{1}$, $L_{2}$ and $L_{3}$ as follows.
\begin{align*}
L_{1}:&=\lambda'(\alpha-1) F^{\alpha}G-\lambda'(\alpha-1) F^{\alpha-1}\sum_{i}\frac{\partial F}{\partial h_{ii}}-(\alpha-1) F^{2\alpha}\sum_{i}\frac{\partial G}{\partial b^{ii}}+(\alpha-1) F^{2\alpha}\\
&\quad +\epsilon \alpha F^{2\alpha-1}\left( F\frac{\partial G}{\partial b^{kl}}b^{kp}b^{pl}-G\sum_{i}\frac{\partial F}{\partial h_{ii}} \right) \\
&=\lambda'(\alpha-1)F^{\alpha-1}(FG-\sum_{i}\frac{\partial F}{\partial h_{ii}})+(\alpha-1)F^{2\alpha}(1-\sum_{i}\frac{\partial G}{\partial b^{ii}}) \\
&\quad +\epsilon \alpha F^{2\alpha-1}\left( F\frac{\partial G}{\partial b^{kl}}b^{kp}b^{pl}-G\sum_{i}\frac{\partial F}{\partial h_{ii}} \right).
\end{align*}
\begin{align*}
L_{2}:&=2\alpha F^{\alpha-1}\frac{\partial F}{\partial h_{ij}}\nabla_{i}F^{\alpha}\nabla_{j}G+\bar{g}(\lambda\partial_{r},\nabla (F^{\alpha}G)) \\
&=2\alpha F^{-1}\frac{\partial F}{\partial h_{ij}}\nabla_{i}F^{\alpha}\nabla_{j}(F^{\alpha}G)-2\alpha F^{-1}G\frac{\partial F}{\partial h_{ij}}\nabla_{i}F^{\alpha}\nabla_{j}F^{\alpha}+\bar{g}(\lambda\partial_{r},\nabla(F^{\alpha}G)) \\
&=2\alpha F^{-1}\frac{\partial F}{\partial h_{ij}}\nabla_{i}F^{\alpha}\nabla_{j}Z+\bar{g}(\lambda\partial_{r},\nabla Z)-2\alpha F^{-1}G\frac{\partial F}{\partial h_{ij}}\nabla_{i}F^{\alpha}\nabla_{j}F^{\alpha}\\
&\quad +2(\alpha-1) F^{-1}\frac{\partial F}{\partial h_{ij}}\nabla_{i}F^{\alpha}\nabla_{j}\Phi+\frac{\alpha-1}{\alpha}\bar{g}(\lambda\partial_{r},\nabla\Phi).
\end{align*}
Substituting
\begin{align*}
\nabla_{i}\Phi=\bar{g}(\lambda\partial_{r},e_{i})=b^{ij}\nabla_{j}F^{\alpha}
\end{align*}
into the above equality, we obtain that
\begin{align*}
L_{2}&=2\alpha F^{-1}\frac{\partial F}{\partial h_{ij}}\nabla_{i}F^{\alpha}\nabla_{j}Z+\bar{g}(\lambda\partial_{r},\nabla Z)-2\alpha F^{-1}G\frac{\partial F}{\partial h_{ij}}\nabla_{i}F^{\alpha}\nabla_{j}F^{\alpha}\\
&\quad +2(\alpha-1) F^{-1}b^{jk}\frac{\partial F}{\partial h_{ij}}\nabla_{i}F^{\alpha}\nabla_{k}F^{\alpha}+\frac{\alpha-1}{\alpha}b^{ik}b^{kj}\nabla_{i}F^{\alpha}\nabla_{j}F^{\alpha}.
\end{align*}
We define
\begin{align*}
L_{3}:&=\alpha F^{2\alpha-1}\frac{\partial F}{\partial h_{ij}}\frac{\partial^{2} G}{\partial b^{kl}\partial b^{pq}}\nabla_{i}b^{pq}\nabla_{j}b^{kl}\\
&\quad +\alpha(\alpha-1)F^{2\alpha-2}\frac{\partial G}{\partial b^{kl}}b^{kp}b^{ql}\nabla_{p}F\nabla_{q}F\\
&\quad +\alpha F^{2\alpha-1}\frac{\partial G}{\partial b^{kl}}b^{kp}b^{ql}\frac{\partial^{2} F}{\partial h_{ij}\partial h_{st}}h_{ijp}h_{stq}\\
&\quad +2\alpha F^{2\alpha-1}\frac{\partial G}{\partial b^{kl}}b^{ks}b^{pt}b^{lq}\frac{\partial F}{\partial h_{ij}}h_{sti}h_{pqj}.
\end{align*}
If $(\frac{\partial G}{\partial b_{kl}})$ is positive definite, from the inverse concavity of $F$, the last two terms of $L_3$ can be estimated as
\begin{align*}
&\alpha F^{2\alpha-1}\frac{\partial G}{\partial b^{kl}}(\frac{\partial^{2}F}{\partial h_{ij}\partial h_{st}}h_{ijp}b^{kp}h_{stq}b^{ql}+2b^{pt}\frac{\partial F}{\partial h_{ij}}b^{ks}h_{sti}b^{lq}h_{pqj})\\
&\qquad \geq 2\alpha F^{2\alpha-2}\frac{\partial G}{\partial b^{kl}}\frac{\partial F}{\partial h_{ij}}\frac{\partial F}{\partial h_{st}}h_{ijp}b^{kp}h_{stq}b^{ql}=\frac{2}{\alpha}\frac{\partial G}{\partial b^{kl}}b^{kp}b^{ql}\nabla_{p}F^{\alpha}\nabla_{q}F^{\alpha}.
\end{align*}
Thus,
\begin{align*}
L_{3}&\geq \alpha F^{2\alpha-1}\frac{\partial F}{\partial h_{ij}}\frac{\partial^{2} G}{\partial b^{kl}\partial b^{pq}}\nabla_{i}b^{pq}\nabla_{j}b^{kl}+\frac{\alpha+1}{\alpha}\frac{\partial G}{\partial b^{kl}}b^{kp}b^{ql}\nabla_{p}F^{\alpha}\nabla_{q}F^{\alpha}.
\end{align*}

Now, we choose $G(b)=\frac{\abs{b}^{2}}{\tr b}$, which is a convex function of $(b^{ij})$ since $G(b)=\sigma_{1}(b)-2\frac{\sigma_{2}(b)}{\sigma_{1}(b)}$. Thus,
\begin{equation*}
\frac{\partial F}{\partial h_{ij}}\frac{\partial^{2} G}{\partial b^{kl}\partial b^{pq}}\nabla_{i}b^{pq}\nabla_{j}b^{kl}\geq 0.
\end{equation*}
We also can check directly that
\begin{align*}
\frac{\partial G}{\partial b^{kl}}=\frac{1}{(\tr b)^{2}}(2b^{kl}\tr b-\abs{b}^{2}\delta_{kl})>0
\end{align*}
when $(b^{ij})$ is a scalar matrix, here ``$>$" means that the corresponding matrix is positive definite.

\begin{lem}
For $G(b)=\frac{\abs{b}^{2}}{\tr b}$ and $\alpha\geq 1$, we have $L_{1}\geq 0$.
\end{lem}

\begin{proof}
First, by using the Cauchy-Schwarz inequality, we have that $\sum_{i}\frac{\partial G}{\partial b^{ii}}=2-\frac{n\abs{b}^{2}}{(\tr b)^{2}}\leq 1$.
We also notice that
\begin{align*}
FG-\sum_{i}\frac{\partial F}{\partial h_{ii}}&=f\cdot \frac{\sum_{i}\kappa_{i}^{-2}}{\sum_{j}\kappa_{j}^{-1}}-\sum_{i}\df^{i}=\frac{1}{\tr b}\sum_{i,j}\df^{i}\kappa_{i}^{2}(\kappa_{i}^{-1}\kappa_{j}^{-2}-\kappa_{i}^{-2}\kappa_{j}^{-1})\\
&=\frac{1}{\tr b}\sum_{i>j}\kappa_{i}^{-2}\kappa_{j}^{-2}(\df^{i}\kappa_{i}^{2}-\df^{j}\kappa_{j}^{2})(\kappa_{i}-\kappa_{j})\geq 0,
\end{align*}
since $F$ is homogeneous of degree $1$ and inverse concave.

Let $(\mu_{1},...,\mu_{n})$ denote the eigenvalues of $(b^{ij})$. Consequently, $\mu_{i}=\frac{1}{\kappa_{i}}$ for $1\leq i\leq n$. Since $F$ and $G$ are both homogeneous functions, we have
\begin{align*}
&F\frac{\partial G}{\partial b^{kl}}b^{km}b^{ml}-G\sum_{i}\frac{\partial F}{\partial h_{ii}} =\sum_{i,j}\left( \frac{\partial F}{\partial h_{ii}}\kappa_{i}\frac{\partial G}{\partial b^{jj}}\mu_{j}^{2}-\frac{\partial F}{\partial h_{ii}}\frac{\partial G}{\partial b^{jj}}\mu_{j} \right)\\
&\qquad =\sum_{i,j}\mu_{i}\mu_{j}\frac{\partial F}{\partial h_{ii}}\kappa_{i}^{2}\frac{\partial G}{\partial b^{jj}}\left( \mu_{j}-\mu_{i} \right) \\
&\qquad =\sum_{i>j}\mu_{i}\mu_{j}\left( \frac{\partial F}{\partial h_{ii}}\kappa_{i}^{2}\frac{\partial G}{\partial b^{jj}}-\frac{\partial F}{\partial h_{jj}}\kappa_{j}^{2}\frac{\partial G}{\partial b^{ii}} \right) \left( \mu_{j}-\mu_{i} \right).
\end{align*}
Assuming $\kappa_{i}\geq \kappa_{j}$, we have $\mu_{j}\geq \mu_{i}$. By using the definition of $G$, it is easy to check that $\frac{\partial G}{\partial b^{jj}}\geq \frac{\partial G}{\partial b^{ii}}$. Combining the above estimates with $\frac{\partial F}{\partial h_{ii}}\kappa_{i}^{2}\geq \frac{\partial F}{\partial h_{jj}}\kappa_{j}^{2}$, we have
\begin{equation*}
\alpha F^{2\alpha-1}\left( F\frac{\partial G}{\partial b^{kl}}b^{km}b^{ml}-G\sum_{i}\frac{\partial F}{\partial h_{ii}} \right)\geq 0.
\end{equation*}
We finish the proof by noticing $\lambda'>0$ and $\epsilon=0$ or $1$ for Euclidean space or the hemisphere.
\end{proof}

Combining the above discussions, we obtain that for $G(b)=\frac{\abs{b}^{2}}{\tr b}$,  we have
\begin{equation}\label{Eq:Z}
\begin{aligned}
&\LF Z-2\alpha F^{-1}\frac{\partial F}{\partial h_{ij}}\nabla_{i}F^{\alpha}\nabla_{j}Z-\bar{g}(\lambda\partial_{r},\nabla Z) \\
&\qquad \geq -2\alpha F^{-1}G\frac{\partial F}{\partial h_{ij}}\nabla_{i}F^{\alpha}\nabla_{j}F^{\alpha}+2(\alpha-1) F^{-1}b^{jk}\frac{\partial F}{\partial h_{ij}}\nabla_{i}F^{\alpha}\nabla_{k}F^{\alpha}\\
&\qquad\quad +\frac{\alpha-1}{\alpha}b^{ik}b^{kj}\nabla_{i}F^{\alpha}\nabla_{j}F^{\alpha}+\frac{\alpha+1}{\alpha}\frac{\partial G}{\partial b^{kl}}b^{kp}b^{ql}\nabla_{p}F^{\alpha}\nabla_{q}F^{\alpha}
\end{aligned}
\end{equation}
if $(\frac{\partial G}{\partial b^{kl}})$ is positive definite.

%%%%%%%%%%%%%%%%%%%%%%%%%%%%%%%%%%%%
\subsection{The case of $\alpha>1$}
\label{Sec:main}
%%%%%%%%%%%%%%%%%%%%%%%%%%%%%%%%%

With out loss of generality, we assume that $\kappa_{1}\leq \kappa_{2}\leq \cdots \leq \kappa_{n}$ and  define $W:=\frac{F^{\alpha}}{\kappa_{1}}-\frac{\alpha-1}{\alpha}\Phi$.
\begin{lem}\label{Th:wmax}
Under the condition of Theorem \ref{main} or Theorem \ref{Th:hemisphere}, if $x_{0}$ is a maximum point of $W$, then $x_{0}$ must be umbilical and $\nabla F^{\alpha}(x_{0})=0$.
\end{lem}

\begin{proof}
We assume that $\kappa_1=\cdots=\kappa_{\mu}<\kappa_{\mu+1}\leq\cdots\leq
\kappa_n$ at  $x_0$. From the proof of Lemma 3.2 in \cite{GaoMa-19}, at $x_{0}$, we have $\nabla_{m} F^{\alpha}=0$ for $2\leq m\leq \mu$ and 
\begin{equation*}
0\geq J_{1}+J_{2}+J_{3},
\end{equation*}
where
\begin{align}\label{eq:J1}
J_{1}=(\alpha-1)\lambda'f^{\alpha-1}\df^{i}(\frac{\kappa_{i}}{\kappa_{1}}-1)+\epsilon \alpha \kappa_{1}^{-1}f^{2\alpha-1}\df^{i} (\frac{\kappa_{i}}{\kappa_{1}}-1),
\end{align}
	
\begin{align*}
J_{2}=\frac{\alpha-1}{\alpha}\kappa_{i}^{-2}(\nabla_{i}F^{\alpha})^{2}-2(\alpha-1)\kappa_{i}^{-1}\frac{\df^{i}}{f}(\nabla_{i}F^{\alpha})^{2}+2\frac{(\alpha-1)^{2}}{\alpha}\kappa_{1}\kappa_{i}^{-2}\frac{\df^{i}}{f}(\nabla_{i}F^{\alpha})^{2}
\end{align*}
and
\begin{align*}
J_{3}&=\alpha f^{2\alpha-1}\kappa_{1}^{-2}\ddf^{ij}h_{ii1}h_{jj1}+\frac{\alpha-1}{\alpha}\kappa_{1}^{-2}(\nabla_{1}F^{\alpha})^{2}+2\alpha f^{2\alpha-1}\kappa_{1}^{-2}\sum_{i>\mu}\df^{i}(\kappa_{i}-\kappa_{1})^{-1}h_{1ii}^{2}\\
&\quad +2\alpha\kappa_{1}^{2}\sum_{i>\mu}(\kappa_{i}-\kappa_{1})^{-1}(\kappa_{1}^{-1}-\frac{\alpha-1}{\alpha}\kappa_{i}^{-1})^{2}\frac{\df^{i}}{f}(\nabla_{i}F^{\alpha})^{2}.
\end{align*}

It is clear that $J_{1}\geq 0$ and the equality occurs if and only if $\kappa_{1}=\kappa_{2}=\cdots =\kappa_{n}$.
From inverse concavity of $F$, we have
\begin{equation*}
\alpha f^{2\alpha-1}\kappa_{1}^{-2}\ddf^{ij}h_{ii1}h_{jj1}\geq -2\alpha f^{2\alpha-1}\kappa_{1}^{-2}\frac{\df^{i}}{\kappa_{i}}h_{ii1}^{2}+\frac{2}{\alpha}\kappa_{1}^{-2}(\nabla_{1}F^{\alpha})^{2}.
\end{equation*}
Combining with $F^{\alpha}\kappa_{1}^{-1}h_{111}=\frac{1}{\alpha}\nabla_{1}F^{\alpha}$ ((3.6) in \cite{GaoMa-19}), we have
\begin{align*}
J_{3}&\geq -\frac{2}{\alpha}\kappa_{1}^{-1}\frac{\df^{1}}{f}(\nabla_{1}F^{\alpha})^{2} +\frac{\alpha+1}{\alpha}\kappa_{1}^{-2}(\nabla_{1}F^{\alpha})^{2}\\
&\quad +2\alpha\kappa_{1}^{2}\sum_{i>\mu}(\kappa_{i}-\kappa_{1})^{-1}(\kappa_{1}^{-1}-\frac{\alpha-1}{\alpha}\kappa_{i}^{-1})^{2}\frac{\df^{i}}{f}(\nabla_{i}F^{\alpha})^{2}.
\end{align*}
Thus,
\begin{align*}
J_{2}+J_{3}&\geq 2(\kappa_{1}^{-1}-\frac{\df^{1}}{f})\kappa_{1}^{-1}(\nabla_{1}F^{\alpha})^{2}+\frac{\alpha-1}{\alpha}\sum_{i>\mu}\kappa_{i}^{-2}(\nabla_{i}F^{\alpha})^{2} \\
&\quad +2\kappa_{1}\sum_{i>\mu}(\kappa_{i}-\kappa_{1})^{-1}(\kappa_{1}^{-1}-\frac{\alpha-1}{\alpha}\kappa_{i}^{-1})\frac{\df^{i}}{f}(\nabla_{i}F^{\alpha})^{2}.
\end{align*}

From $\alpha\geq 1$ and $f>\df^{1}\kappa_{1}$, we have $J_{2}+J_{3}\geq 0$. This implies that $J_{1}=0$ and $J_{2}+J_{3}=0$. We finish the proof consequently.
\end{proof}

\subsection{Proof of Theorem \ref{main}}\label{sec3.3}
\begin{proof}
Let $x_{0}\in M^n$ be a maximum point of $W$. We know that $x_{0}$ is a umbilical point from Lemma \ref{Th:wmax}. Thus, there exists a neighborhood $U$ of $x_{0}$ such that $(\frac{\partial G}{\partial b^{kl}})$ is positive definite. Hence inequality \eqref{Eq:Z} holds in $U$. At $x_{0}$, the right-hand side of \eqref{Eq:Z} can be written as
\begin{align*}
&-2\alpha f^{-1}\kappa_{1}^{-1}\df^{i}(\nabla_{i} F^{\alpha})^{2}+2(\alpha-1) f^{-1}\kappa_{i}^{-1}\df^{i}(\nabla_{i} F^{\alpha})^{2}\\
&\qquad\quad +\frac{\alpha-1}{\alpha}\kappa_{i}^{-2}(\nabla_{i} F^{\alpha})^{2}+\frac{\alpha+1}{n\alpha}\kappa_{i}^{-2}(\nabla_{i} F^{\alpha})^{2}.
\end{align*}
The coefficient of $(\nabla_{i} F^{\alpha})^{2}$ is
\begin{align*}
&-2\alpha f^{-1}\kappa_{1}^{-1}\df^{i}+2(\alpha-1) f^{-1}\kappa_{i}^{-1}\df^{i}+\frac{\alpha-1}{\alpha}\kappa_{i}^{-2}+\frac{\alpha+1}{n\alpha}\kappa_{i}^{-2}\\
&\qquad =\frac{(n-1)(\alpha-1)}{n\alpha}\kappa_{1}^{-2}>0,
\end{align*}
since $n\kappa_{1}^{-1}\df^{i}=\kappa_{1}^{-2}\sum_{i}\df^{i}\kappa_{i}=\kappa_{1}^{-2}f$. Thus, by choosing a smaller neighborhood $U$ of $x_{0}$ (still denoted by $U$), we have
\begin{align*}
\LF Z-2\alpha F^{-1}\frac{\partial F}{\partial h_{ij}}\nabla_{i}F^{\alpha}\nabla_{j}Z-\bar{g}(\lambda\partial_{r},\nabla Z)\geq 0.
\end{align*}
Since $Z(x)\leq W(x)\leq W(x_{0})=Z(x_{0})$ for all $x\in M^n$, we know that $x_{0}$ is a maximum point of $Z$. Therefore, by applying the strong maximum principle, $Z$ is constant in $U$, which implies  immediately that $W$ is constant in $U$. Since $M^n$ is connected, we obtain that $W$ is constant in $M^n$. Thus, all points of $M^n$ are umbilical. In Euclidean space, this implies that $M^n$ is a sphere.
\end{proof}

%%%%%%%%%%%%%%%%%%%%%%%%%%%%%%%%%%%
\section{Proof of Theorem \ref{Th:hemisphere}}
\label{Sec:hemisphere}
%%%%%%%%%%%%%%%%%%%%%%%%%%%%%%%%%%%%%

Proof of Theorem \ref{main} also works for Theorem \ref{Th:hemisphere} when $\alpha>1$. From Lemma \ref{Th:wmax}, the fact that $W$ is constant implies that $\bar{g}(\lambda\partial_{r},e_{i})=b^{ij}\nabla_{j}F^{\alpha}=0$. This means that $M^n$ is a slice $\{ r_{0} \}\times \mathbb{S}^{n}$ in $\mathbb{S}^{n+1}_{+}$.

Now we consider the case of $\alpha=1$. Without loss of generality, we assume that $F(I)=1$, where $I$ is the identity matrix. The auxiliary quantity $Z=FG$ can be written as $Z=\frac{g(\mu)}{f_{*}(\mu)}$ where $\mu=(\mu_{1},...,\mu_{n})=(\frac{1}{\kappa_{1}},...,\frac{1}{\kappa_{n}})$. $F$ is inverse concave means $f_{*}(\mu)$ is concave, and we know $g(\mu)=\frac{\sum_{i}\mu_{i}^{2}}{\sum_{j}\mu_{j}}$ is convex. $f_*$ and $g$ are both homogeneous of degree $1$ and $f_{*}(1,...,1)=g(1,...,1)=1$. From Lemma \ref{lem-conc}, we know $f_{*}\leq \sum_{j}\mu_{j}/n\leq g$. Therefore, $Z\geq 1$.
On the other hand, if $x_{0}$ is a maximum point of $W$, then $x_{0}$ is also a maximum point of $Z$. Combining with Lemma \ref{Th:wmax}, we know $Z_{\max}=Z(x_{0})=1$. This implies $Z\equiv 1$. Consequently, $W\equiv 1$. Thus we finish the proof.

%%%%%%%%%%%%%%%%%%%%%%%%%%	
% ------------------------------------------------------------------------

\section{An alternative Proof of Theorem \ref{main}  and Theorem \ref{Th:hemisphere}}\label{sec:proof2}

In this section, we give an alternative proof of Theorem \ref{main}  and Theorem \ref{Th:hemisphere}.

Let $N=[0,\bar{r}	)\times \mathbb{S}^{n}$ be a warped product manifold with metric $\bar{g}=dr\otimes dr+\lambda^{2}(r)\sigma$ which has constant sectional curvature $\epsilon\geq 0$. Let $M^n$ be a closed strictly convex hypersurface in $N$.
We define a tensor on $M^n$ by
\begin{equation}\label{Tij}
T_{kl}= F^\alpha  b^{kl}-\frac{\alpha-1}{\alpha}\Phi g^{kl}-\beta g^{kl},
\end{equation}
where $(b^{kl})$  is the inverse matrix of $(h_{kl})$, $\Phi(r)=\int_{0}^{r}\lambda(s)ds$, $g^{kl}$ is the inverse of the metric and $\beta$ is a constant to be determined.
Recall that $\mathcal{L}$ denotes the operator $\alpha F^{\alpha-1}\frac{\partial F}{\partial h_{ij}}\nabla_{i}\nabla_{j}$, we have the following proposition.
\begin{prop}\label{prop3.1}
Let $M^n$ be a closed strictly convex hypersurfaces in $N$ such that \eqref{Eq:ssN} holds, then we have the following equation.
\begin{equation}\label{LTkl}
\begin{aligned}
\mathcal{L}T_{kl}=&(\alpha-1)\lambda' F^\alpha  b^{kl}-(\alpha-1)\lambda'F^{\alpha-1}\sum_i\frac{\partial F}{\partial h_{ii}}\delta_{kl}\\
&+\bar{g}(\lambda\partial_r,\nabla b^{kl})  F^\alpha +\bar{g}(\lambda\partial_r,\nabla  F^\alpha )b^{kl}\\
&+\alpha(\alpha-1)F^{2\alpha-2}b^{kp}b^{ql}\nabla_p F\nabla_qF+\alpha F^{2\alpha-1}b^{kp}b^{ql}\frac{\partial^2F}{\partial h_{ij}\partial h_{st}}h_{ijp}h_{stq}\\
&+2\alpha F^{2\alpha-1}b^{ks}b^{pt}b^{lq}\frac{\partial F}{\partial h_{ij}}h_{sti}h_{pqj}+2\alpha^2 F^{2\alpha-2}\frac{\partial F}{\partial h_{ij}}\nabla_i F\nabla_j b^{kl}\\
&-\epsilon\alpha F^{2\alpha-1}b^{kp}b^{ql}\frac{\partial F}{\partial h_{ij}}(h_{iq}\delta_{pj}-h_{ij}\delta_{pq}+h_{pq}\delta_{ij}-h_{pj}\delta_{iq}).
\end{aligned}
\end{equation}
\end{prop}

\begin{proof}
By differentiating \eqref{Tij} twice, we have
\begin{equation}\label{dtij}
\begin{aligned}
\nabla_i\nabla_j T_{kl}=&\nabla_i(\nabla_j F^{\alpha} b^{kl}+ F^\alpha  \nabla_j b^{kl}-\frac{\alpha-1}{\alpha}\nabla_j\Phi g^{kl})\\
=&\nabla_i\nabla_j F^\alpha b^{kl}+\nabla_i F^\alpha \nabla_jb^{kl}+\nabla_j F^\alpha \nabla_ib^{kl}\\
&+ F^\alpha \nabla_i\nabla_jb^{kl}-\frac{\alpha-1}{\alpha}\nabla_i\nabla_j\Phi g^{kl},
\end{aligned}
\end{equation}
which implies that
\begin{equation}\label{dtij2}
\begin{aligned}
\mathcal{L} T_{kl}=&\mathcal{L}  F^\alpha  b^{kl}+ F^\alpha  \mathcal{L} b^{kl}+2\alpha^2 F^{2\alpha-2}\frac{\partial F}{\partial h_{ij}}\nabla_iF\nabla_j b^{kl}\\
&-\frac{\alpha-1}{\alpha}\mathcal {L}\Phi g^{kl}.
\end{aligned}
\end{equation}
Substituting equations (1),(3) and (6) in Proposition \ref{Prop:formula} into \eqref{dtij2}, we obtain that
\begin{equation}\label{dtij3}
\begin{aligned}
\mathcal{L} T_{kl}=&\Big(\bar{g}( \lambda\partial_{r}, \nabla F^{\alpha})+\alpha\lambda' F^{\alpha}-\alpha F^{2\alpha-1}\frac{\partial F}{\partial h_{ij}}h_{jm}h_{mi}\Big) b^{kl}\\
&+ F^\alpha  \Big(\bar{g}(\lambda\partial_{r},\nabla b^{kl})-\lambda'b^{kl}+\alpha(\alpha-1)F^{\alpha-2}b^{kp}b^{ql}\nabla_{p}F\nabla_{q}F\\
&\quad +\alpha F^{\alpha-1}b^{kp}b^{ql}\frac{\partial^{2} F}{\partial h_{ij}\partial h_{st}}h_{ijp}h_{stq}+\alpha F^{\alpha-1}b^{kl}\frac{\partial F}{\partial h_{ij}}h_{mj}h_{mi}\\
&\quad -(\alpha-1) F^{\alpha}\delta_{kl}+2\alpha F^{\alpha-1}b^{ks}b^{pt}b^{lq}\frac{\partial F}{\partial h_{ij}}h_{sti}h_{pqj} \\
&\quad -\epsilon \alpha F^{\alpha-1} b^{kp}b^{ql} \frac{\partial F}{\partial h_{ij}}\big(h_{iq}\delta_{pj}-h_{ij}\delta_{pq}+h_{pq}\delta_{ij}-h_{pj}\delta_{iq}\big)\Big)\\
&+2\alpha^2 F^{2\alpha-2}\frac{\partial F}{\partial h_{ij}}\nabla_iF\nabla_j b^{kl}\\
&-\frac{\alpha-1}{\alpha}\Big(\lambda'\alpha F^{\alpha-1}\sum_{i}\frac{\partial F}{\partial h_{ii}}-\alpha F^{2\alpha}\Big)g^{kl},
\end{aligned}
\end{equation}
from which we immediately derive \eqref{LTkl}.
\end{proof}

\begin{proof}[Proof of Theorem \ref{main} and Theorem \ref{Th:hemisphere}]
We denote the right-hand side of \eqref{LTkl} by $N_{kl}$, i.e.,
\begin{equation}\label{Nkl}
\begin{aligned}
N_{kl}=&(\alpha-1)\lambda' F^\alpha  b^{kl}-(\alpha-1)\lambda'F^{\alpha-1}\sum_i\frac{\partial F}{\partial h_{ii}}\delta_{kl}\\
&+\bar{g}(\lambda\partial_r,\nabla b^{kl})  F^\alpha +\bar{g}(\lambda\partial_r,\nabla  F^\alpha )b^{kl}\\
&+\alpha(\alpha-1)F^{2\alpha-2}b^{kp}b^{ql}\nabla_p F\nabla_qF+\alpha F^{2\alpha-1}b^{kp}b^{ql}\frac{\partial^2F}{\partial h_{ij}\partial h_{st}}h_{ijp}h_{stq}\\
&+2\alpha F^{2\alpha-1}b^{ks}b^{pt}b^{lq}\frac{\partial F}{\partial h_{ij}}h_{sti}h_{pqj}+2\alpha^2 F^{2\alpha-2}\frac{\partial F}{\partial h_{ij}}\nabla_i F\nabla_j b^{kl}\\
&-\epsilon\alpha F^{2\alpha-1}b^{kp}b^{ql}\frac{\partial F}{\partial h_{ij}}(h_{iq}\delta_{pj}-h_{ij}\delta_{pq}+h_{pq}\delta_{ij}-h_{pj}\delta_{iq}).
\end{aligned}
\end{equation}

We consider $T$ as a function on unit tangent bundle $UTM^n$, i.e., $x\in M^n$, $v$ is a unit vector in
$T_xM^n$ and $T(x,v)=T_{ij}(x)v^iv^j$.
Since $M^n$ is compact, $T$ attains a minimum and a maximum on $UTM^n$. We choose $\beta$ such that $T_{ij}\leq 0$ everywhere on $M^n$ and $T_{ij}v^iv^j=0$ at some point $x_{0}\in M^n$ for some unit tangent vector $v$. We can obtain a contradiction to the maximum principle applied to $\eqref{LTkl}$ by estimating the derivative terms
as that done in \cite{BAndrews07} (see also \cite{A-M-Z}). As in \cite{McCoy11}, it suffices to show that at $x_{0}$, we have
\begin{equation}\label{goal}
N_{kl}v^kv^l-\inf_{\Gamma} 2\alpha F^{\alpha-1}\frac{\partial F}{\partial h_{kl}}(2\Gamma^p_k\nabla_lT_{ip}v^i-\Gamma^p_k\Gamma^q_l T_{pq})\geq 0,
\end{equation}
where $N_{kl}$ is defined in $\eqref{Nkl}$.

We can choose coordinates at  $x_0$ such that $v=e_1$. Since $T\leq 0$ and $(x_{0},v)$ is a maximum point of $T$, we have 
\begin{equation}\label{critical}
T_{1k}=0,\nabla_k T_{11}=0,~\text{for any}~k=1,\ldots,n.
\end{equation}
As $T_{1k}=0,~\text{for any} ~k=1,\ldots,n$, $e_1$ is an eigenvector of $T_{kl}$ with eigenvalue $0$. We can further choose coordinates at  $x_0$ such that $h_{ij}=\kappa_i\delta_{ij}$ and $g_{ij}=\delta_{ij}$.
From \eqref{critical}, we obtain the following relations.
\begin{align}
0=T_{11}&= F^\alpha \kappa_1^{-1}-\frac{\alpha-1}{\alpha}\Phi-\beta,\label{lambda1}\\
0\geq T_{pp}&= F^\alpha \kappa_p^{-1}-\frac{\alpha-1}{\alpha}\Phi-\beta,~p=2,\ldots,n.\label{lambdap}
\end{align}
From \eqref{lambda1} and \eqref{lambdap}, we know that $\kappa_p\geq\kappa_1$ for $p=2,\ldots,n$. With out loss of generality, 
we can choose coordinates such that  $\kappa_1=\cdots=\kappa_{\mu}<\kappa_{\mu+1}\leq\cdots\leq
\kappa_n$ at  $x_0$. By using \eqref{lambda1} and \eqref{lambdap} again, we have
\begin{equation}\label{tpp}
	T_{mm}=0,~1\leq m\leq \mu;~	0>T_{pp}= F^\alpha (\kappa_p^{-1}-\kappa_1^{-1})=\frac{ F^\alpha }{\kappa_p\kappa_1}(\kappa_1-\kappa_p),~p>\mu.
\end{equation}
We also have that
\begin{equation}\label{tk1m}
\nabla_kT_{1m}=0,~2\leq m\leq \mu,~k=1,\ldots,n.
\end{equation}
In fact, if there exists some $2\leq m \leq\mu$ and $1\leq k\leq n$ such that $\nabla_kT_{1m}\neq 0$, we can easily choose $\Gamma$ such that
the left-hand side of $\eqref{goal}$ is positive, which contradicts the property that $(x_0,v=e_1)$ is a maximum point of $T$.

By using the definition of $T_{ij}$, we have
\begin{equation*}
	\nabla_kT_{ij}=\nabla_k F^\alpha b^{ij}+ F^\alpha \nabla_kb^{ij}-\frac{\alpha-1}{\alpha}\nabla_k\Phi g^{ij}.
\end{equation*}
Since $g_{ij}=\delta_{ij},~h_{ij}=\kappa_i\delta_{ij}$ at  $x_0$, we have
\begin{equation}\label{Tk1p}
	\begin{aligned}
		\nabla_kT_{1p}= F^\alpha \nabla_kb^{1p}=- F^\alpha \kappa_1^{-1}\kappa_p^{-1}h_{1pk},~p\geq 2,~k=1,\ldots,n,
	\end{aligned}
\end{equation}
and 
\begin{equation*}
	\begin{aligned}
		0=\nabla_kT_{11}=\nabla_k F^\alpha \kappa_1^{-1}+  F^\alpha  \nabla_kb^{11}-\frac{\alpha-1}{\alpha}\nabla_k\Phi,~k=1,\ldots,n,
	\end{aligned}
\end{equation*}
which implies that
\begin{equation}\label{t11k}
 F^\alpha \nabla_kb^{11}=\frac{\alpha-1}{\alpha}\nabla_k\Phi-\nabla_k F^\alpha \kappa_1^{-1}
\end{equation}
at  $x_0$. On the other hand, by using the definition of $\Phi$, we have
\begin{equation}\label{phik}
\nabla_k\Phi=\bar{g}(\lambda\partial_r,e_k).
\end{equation}
By using the self-similar equation \eqref{maine}, we have the following relation at  $x_0$.
\begin{equation}\label{fk}
	\nabla_k F^\alpha =\bar{g}(\lambda\partial_r,h_{kl}e_l)
	=\kappa_k\bar{g}(\lambda\partial_r,e_k).
\end{equation}	
\eqref{phik} and \eqref{fk} implies that
\begin{equation}\label{phif}
	\nabla_k\Phi=\kappa_k^{-1}\nabla_k F^\alpha =\kappa_k^{-1}\alpha F^{\alpha-1}\nabla_kF.
\end{equation}
Combining the equations \eqref{t11k} and \eqref{phif}, we have
\begin{equation}\label{b11k}
	\nabla_kb^{11}=(\frac{\alpha-1}{\alpha}\kappa_k^{-1}-\kappa_1^{-1})\alpha F^{-1}\nabla_kF.
\end{equation}	
Since $\nabla_kb^{11}=-\kappa_1^{-2}h_{11k}$ at  $x_0$, we have
\begin{equation}\label{h11k}
	h_{11k}=-\kappa_1^2(\frac{\alpha-1}{\alpha}\kappa_k^{-1}-\kappa_1^{-1})\alpha F^{-1}\nabla_kF.
\end{equation}	
Using \eqref{tk1m} and \eqref{Tk1p}, we have
$
		0=\nabla_kT_{1m}=- F^\alpha \kappa_1^{-1}\kappa_m^{-1}h_{1mk},~2\leq m\leq \mu,~k=1,\ldots,n,
$
which in combination with \eqref{b11k} and \eqref{h11k} implies the following information at $x_0$.
\begin{equation}\label{info}
	h_{1mk}=0,~\nabla_mb^{11}=0,~\nabla_mF=0,~2\leq m\leq \mu,~k=1,\ldots,n,
\end{equation}

Recall that $v=e_1$, by using \eqref{tpp} and \eqref{tk1m}, to prove $\eqref{goal}$, it suffices to prove that
\begin{equation}\label{goal2}
Q:=N_{11}-\inf_{\Gamma}~ 2\alpha F^{\alpha-1}\frac{\partial F}{\partial h_{kk}}\sum_{p>\mu}(2\Gamma^p_k\nabla_kT_{1p}-(\Gamma^p_k)^2T_{pp})\geq 0.
\end{equation}
By choosing $\Gamma^p_k=\frac{\nabla_kT_{1p}}{T_{pp}}$ and using \eqref{tpp}, the infimum over $\Gamma$ in \eqref{goal2} can be computed as follows.
\begin{equation}\label{inf}
\begin{aligned}
\inf_{\Gamma}& ~2\alpha F^{\alpha-1}\frac{\partial F}{\partial h_{kk}}\sum_{p>\mu}(2\Gamma^p_k\nabla_kT_{1p}-(\Gamma^p_k)^2T_{pp})\\
&=2\alpha F^{\alpha-1}\frac{\partial F}{\partial h_{kk}}\sum_{p>\mu}\frac{\kappa_p\kappa_1}{ F^\alpha (\kappa_1-\kappa_p)}(\nabla_kT_{1p})^2\\
&=2\alpha F^{\alpha-1}\frac{\partial F}{\partial h_{kk}}\sum_{p>\mu}\frac{\kappa_p\kappa_1}{ F^\alpha (\kappa_1-\kappa_p)}F^{2\alpha}\kappa_1^{-2}\kappa_p^{-2}h_{1pk}^2\\
&=2\alpha F^{2\alpha-1}\frac{\partial F}{\partial h_{kk}}\sum_{p>\mu}\frac{\kappa_p^{-1}\kappa_1^{-1}}{\kappa_1-\kappa_p}h_{1pk}^2,
\end{aligned}
\end{equation}
where we used \eqref{Tk1p} in the second equality.

In view of \eqref{Nkl} and \eqref{inf}, recall that  $g_{ij}=\delta_{ij}$ and $h_{ij}=\kappa_i\delta_{ij}$ at  $x_0$,
we have
\begin{equation}\label{Q}
\begin{aligned}
Q=&(\alpha-1)\lambda' F^\alpha  \kappa_1^{-1}-(\alpha-1)\lambda'F^{\alpha-1}\sum_i\frac{\partial F}{\partial h_{ii}}\\
&+\bar{g}(\lambda\partial_r,\nabla b^{11})  F^\alpha +\bar{g}(\lambda\partial_r,\nabla  F^\alpha )\kappa_1^{-1}\\
&+\alpha(\alpha-1)F^{2\alpha-2}\kappa_1^{-2}(\nabla_1 F)^2+\alpha F^{2\alpha-1}\kappa_1^{-2}\frac{\partial^2F}{\partial h_{ij}\partial h_{st}}h_{ij1}h_{st1}\\
&+2\alpha F^{2\alpha-1}\kappa_1^{-2}\kappa_p^{-1}\frac{\partial F}{\partial h_{ii}}h_{1pi}^2+2\alpha^2 F^{2\alpha-2}\frac{\partial F}{\partial h_{ii}}\nabla_i F\nabla_i b^{11}\\
&-\epsilon\alpha F^{2\alpha-1}\kappa_1^{-2}\frac{\partial F}{\partial h_{ii}}(\kappa_1-\kappa_i)-2\alpha F^{2\alpha-1}\frac{\partial F}{\partial h_{kk}}\sum_{p>\mu}\frac{\kappa_p^{-1}\kappa_1^{-1}}{\kappa_1-\kappa_p}h_{1pk}^2\\
=&Q_1+Q_2+Q_3,
\end{aligned}
\end{equation}
with $Q_1,~Q_2,~Q_3$ defined by
\begin{equation}\label{Q123}
	\begin{aligned}
		Q_1:=&(\alpha-1)\lambda'F^{\alpha-1}\sum_m\frac{\partial F}{\partial h_{mm}}(\frac{\kappa_m}{\kappa_1}-1)+\epsilon\alpha F^{2\alpha-1}\kappa_1^{-1}\sum_m\frac{\partial F}{\partial h_{mm}}(\frac{\kappa_m}{\kappa_1}-1),\\
		Q_2:=&\bar{g}(\lambda\partial_r,\nabla b^{11})  F^\alpha +\bar{g}(\lambda\partial_r,\nabla  F^\alpha )\kappa_1^{-1}\\
		&+2\alpha F^{2\alpha-1}\kappa_1^{-3}\frac{\partial F}{\partial h_{ii}}h_{11i}^2+2\alpha^2 F^{2\alpha-2}\frac{\partial F}{\partial h_{ii}}\nabla_i F\nabla_i b^{11},\\
		Q_3:=&\alpha(\alpha-1)F^{2\alpha-2}\kappa_1^{-2}(\nabla_1 F)^2+\alpha F^{2\alpha-1}\kappa_1^{-2}\frac{\partial^2F}{\partial h_{ij}\partial h_{st}}h_{ij1}h_{st1}\\
		&+2\alpha F^{2\alpha-1}\sum_{p>\mu}\frac{\kappa_1^{-2}}{\kappa_p-\kappa_1}\frac{\partial F}{\partial h_{kk}}h_{1pk}^2.
	\end{aligned}
\end{equation}
We estimate the terms in \eqref{Q123} as follows.
First, since $\alpha\geq 1$ and $\kappa_1\leq \kappa_m,~1\leq m\leq n$, we have $Q_1\geq 0$. Substituting \eqref{b11k} and \eqref{h11k} into the expression of $Q_2$, using \eqref{t11k}, \eqref{fk} and \eqref{info}, we can rewrite $Q_2$
in terms of $\nabla_kF$.
\begin{equation}\label{Q2final}
	\begin{aligned}
	Q_2=&\alpha(\alpha-1)\kappa_k^{-2}F^{2\alpha-2}(\nabla_kF)^2\\&+2\alpha^3F^{2\alpha-3}\kappa_1\frac{\partial F}{\partial h_{kk}}(\frac{\alpha-1}{\alpha}\kappa_k^{-1}-\kappa_1^{-1})^2(\nabla_kF)^2	\\
&+2\alpha^3F^{2\alpha-3}\frac{\partial F}{\partial h_{kk}}(\frac{\alpha-1}{\alpha}\kappa_k^{-1}-\kappa_1^{-1})(\nabla_kF)^2\\
=&\alpha(\alpha-1)\kappa_1^{-2}F^{2\alpha-2}(\nabla_1F)^2+\alpha(\alpha-1)\sum_{p>\mu}\kappa_p^{-2}F^{2\alpha-2}(\nabla_pF)^2\\&+2\alpha F^{2\alpha-3}\kappa_1^{-1}\frac{\partial F}{\partial h_{11}}(\nabla_1F)^2+2\alpha^3F^{2\alpha-3}\kappa_1\sum_{p>\mu}\frac{\partial F}{\partial h_{pp}}(\frac{\alpha-1}{\alpha}\kappa_p^{-1}-\kappa_1^{-1})^2(\nabla_pF)^2	\\
&-2\alpha^2 F^{2\alpha-3}\kappa_1^{-1}\frac{\partial F}{\partial h_{11}}(\nabla_1F)^2+2\alpha^3F^{2\alpha-3}\sum_{p>\mu}\frac{\partial F}{\partial h_{pp}}(\frac{\alpha-1}{\alpha}\kappa_p^{-1}-\kappa_1^{-1})(\nabla_pF)^2.
	\end{aligned}
\end{equation}
Applying \eqref{F-dd}, we have
\begin{equation}\label{Q3t1}
	\begin{aligned}
	\frac{\partial^2F}{\partial h_{ij}\partial h_{st}}h_{ij1}h_{st1}
	=&\frac{\partial F}{\partial h_{ij}}h_{ii1}h_{jj1}+2\sum_{i>j}\frac{1}{\kappa_i-\kappa_j}(\frac{\partial F}{\partial h_{ii}}-\frac{\partial F}{\partial h_{jj}})h_{ij1}^2\\
	=&\frac{\partial F}{\partial h_{ij}}h_{ii1}h_{jj1}+2\sum_{p>\mu}\frac{1}{\kappa_p-\kappa_1}(\frac{\partial F}{\partial h_{pp}}-\frac{\partial F}{\partial h_{11}})h_{p11}^2\\
	&+2\sum_{p>k>\mu}\frac{1}{\kappa_p-\kappa_k}(\frac{\partial F}{\partial h_{pp}}-\frac{\partial F}{\partial h_{kk}})h_{pk1}^2.
	\end{aligned}
\end{equation}
We also have
\begin{equation}
	\begin{aligned}\label{Q3t2}
	\sum_{p>\mu}&\frac{2}{\kappa_p-\kappa_1}\frac{\partial F}{\partial h_{kk}}h_{1pk}^2\\
	=&\sum_{p>\mu}	\frac{2}{\kappa_p-\kappa_1}\frac{\partial F}{\partial h_{11}}h_{1p1}^2+\sum_{p>\mu}	\frac{2}{\kappa_p-\kappa_1}\frac{\partial F}{\partial h_{pp}}h_{1pp}^2\\
	&+\sum_{p>k>\mu}\frac{2}{\kappa_p-\kappa_1}\frac{\partial F}{\partial h_{kk}}h_{1pk}^2+\sum_{p>k>\mu}\frac{2}{\kappa_k-\kappa_1}\frac{\partial F}{\partial h_{pp}}h_{1pk}^2.	
	\end{aligned}
\end{equation}
Combing \eqref{Q3t1} and \eqref{Q3t2},
we can simplify  and estimate $Q_3$ as follows.
\begin{equation*}
	\begin{aligned}
		Q_3=&\alpha(\alpha-1)F^{2\alpha-2}\kappa_1^{-2}(\nabla_1 F)^2\\
		&+\alpha F^{2\alpha-1}\kappa_1^{-2}(\frac{\partial^2F}{\partial h_{ij}\partial h_{st}}h_{ij1}h_{st1}+\sum_{p>\mu}\frac{2}{\kappa_p-\kappa_1}\frac{\partial F}{\partial h_{kk}}h_{1pk}^2)\\
		=&\alpha(\alpha-1)F^{2\alpha-2}\kappa_1^{-2}(\nabla_1 F)^2\\
		&+\alpha F^{2\alpha-1}\kappa_1^{-2}\Big(\frac{\partial F}{\partial h_{ij}}h_{ii1}h_{jj1}+\sum_{p>\mu}	\frac{2}{\kappa_p-\kappa_1}\frac{\partial F}{\partial h_{pp}}h_{1pp}^2+\sum_{p>\mu}\frac{2}{\kappa_p-\kappa_1}\frac{\partial F}{\partial h_{pp}}h_{p11}^2\\
		&\hspace{6em}+2\sum_{p>k>\mu}(\frac{1}{\kappa_p-\kappa_k}(\frac{\partial F}{\partial h_{pp}}-\frac{\partial F}{\partial h_{kk}})+\frac{1}{\kappa_p-\kappa_1}\frac{\partial F}{\partial h_{kk}}+\frac{1}{\kappa_k-\kappa_1}\frac{\partial F}{\partial h_{pp}})h_{1pk}^2\Big)\\
		\geq &\alpha(\alpha-1)F^{2\alpha-2}\kappa_1^{-2}(\nabla_1 F)^2\\
		&+\alpha F^{2\alpha-1}\kappa_1^{-2}\Big(\frac{\partial F}{\partial h_{ij}}h_{ii1}h_{jj1}+\sum_{p>\mu}	\frac{2}{\kappa_p-\kappa_1}\frac{\partial F}{\partial h_{pp}}h_{1pp}^2+\sum_{p>\mu}\frac{2}{\kappa_p-\kappa_1}\frac{\partial F}{\partial h_{pp}}h_{p11}^2\Big)\\
		\geq &\alpha(\alpha-1)F^{2\alpha-2}\kappa_1^{-2}(\nabla_1 F)^2\\
		&+\alpha F^{2\alpha-1}\kappa_1^{-2}\Big(2F^{-1}(\frac{\partial F}{\partial h_{kk}}h_{kk1})^2-2\sum_{m\geq 1}\kappa_m^{-1}\frac{\partial F}{\partial h_{mm}}h_{mm1}^2\\
		&\hspace{6em}+\sum_{p>\mu}	\frac{2}{\kappa_p-\kappa_1}\frac{\partial F}{\partial h_{pp}}h_{pp1}^2+\sum_{p>\mu}\frac{2}{\kappa_p-\kappa_1}\frac{\partial F}{\partial h_{pp}}h_{11p}^2\Big)\\
			\geq &\alpha(\alpha-1)F^{2\alpha-2}\kappa_1^{-2}(\nabla_1 F)^2\\
			&+\alpha F^{2\alpha-1}\kappa_1^{-2}\Big(2F^{-1}(\nabla_1F)^2-2\kappa_1^{-1}\frac{\partial F}{\partial h_{11}}h_{111}^2+2\sum_{p>\mu}\frac{1}{\kappa_p-\kappa_1}\frac{\partial F}{\partial h_{pp}}h_{11p}^2\Big),
	\end{aligned}
\end{equation*}
where we used \eqref{s2:inv-conc} and \eqref{inv-conc-1} in the first inequality and the second inequality, respectively.
Substituting \eqref{h11k} into the above expression, we can estimate $Q_3$
in terms of $\nabla_kF$.
\begin{equation}\label{Q3final}	
\begin{aligned}
		Q_3\geq&\alpha(\alpha-1)F^{2\alpha-2}\kappa_1^{-2}(\nabla_1 F)^2\\
		&+\alpha F^{2\alpha-1}\kappa_1^{-2}\Big(2F^{-1}(\nabla_1F)^2-2\kappa_1^{-1}\frac{\partial F}{\partial h_{11}}(\kappa_1 F^{-1}\nabla_1F)^2\\
		&\hspace{6em}+2\sum_{p>\mu}\frac{1}{\kappa_p-\kappa_1}\frac{\partial F}{\partial h_{pp}}(-\kappa_1^2(\frac{\alpha-1}{\alpha}\kappa_p^{-1}-\kappa_1^{-1})\alpha F^{-1}\nabla_pF)^2\Big)\\
		=&\alpha(\alpha-1)F^{2\alpha-2}\kappa_1^{-2}(\nabla_1 F)^2\\
		&+\alpha F^{2\alpha-1}\kappa_1^{-2}\Big(2(F-\frac{\partial F}{\partial h_{11}}\kappa_1)F^{-2}(\nabla_1F)^2\\
		&\hspace{6em}+2\sum_{p>\mu}\frac{1}{\kappa_p-\kappa_1}\frac{\partial F}{\partial h_{pp}}\kappa_1^4(\frac{\alpha-1}{\alpha}\kappa_p^{-1}-\kappa_1^{-1})^2\alpha^2 F^{-2}(\nabla_pF)^2\Big)\\
		=&\alpha F^{2\alpha-2}\kappa_1^{-2}(\alpha+1-2\frac{\partial F}{\partial h_{11}}\kappa_1F^{-1})(\nabla_1F)^2\\
		&+2\sum_{p>\mu}\alpha^3F^{2\alpha-3}\kappa_1^2(\kappa_p-\kappa_1)^{-1}\frac{\partial F}{\partial h_{pp}}(\frac{\alpha-1}{\alpha}\kappa_p^{-1}-\kappa_1^{-1})^2(\nabla_pF)^2.
	\end{aligned}
\end{equation}
Combing \eqref{Q2final} and \eqref{Q3final}, we obtain that
\begin{equation}
	\begin{aligned}\label{Q2Q3}
		Q_2+Q_3\geq &\alpha(\alpha-1)\kappa_1^{-2}F^{2\alpha-2}(\nabla_1F)^2+\alpha(\alpha-1)\sum_{p>\mu}\kappa_p^{-2}F^{2\alpha-2}(\nabla_pF)^2\\&+2\alpha F^{2\alpha-3}\kappa_1^{-1}\frac{\partial F}{\partial h_{11}}(\nabla_1F)^2+2\alpha^3F^{2\alpha-3}\kappa_1\sum_{p>\mu}\frac{\partial F}{\partial h_{pp}}(\frac{\alpha-1}{\alpha}\kappa_p^{-1}-\kappa_1^{-1})^2(\nabla_pF)^2	\\
		&-2\alpha^2 F^{2\alpha-3}\kappa_1^{-1}\frac{\partial F}{\partial h_{11}}(\nabla_1F)^2+2\alpha^3F^{2\alpha-3}\sum_{p>\mu}\frac{\partial F}{\partial h_{pp}}(\frac{\alpha-1}{\alpha}\kappa_p^{-1}-\kappa_1^{-1})(\nabla_pF)^2\\
		&+\alpha F^{2\alpha-2}\kappa_1^{-2}(\alpha+1-2\frac{\partial F}{\partial h_{11}}\kappa_1F^{-1})(\nabla_1F)^2\\
		&+2\sum_{p>\mu}\alpha^3F^{2\alpha-3}\kappa_1^2(\kappa_p-\kappa_1)^{-1}\frac{\partial F}{\partial h_{pp}}(\frac{\alpha-1}{\alpha}\kappa_p^{-1}-\kappa_1^{-1})^2(\nabla_pF)^2\\
		=&	2\alpha^2F^{2\alpha-3}\kappa_1^{-2}(F-\frac{\partial F}{\partial h_{11}}\kappa_1)(\nabla_1F)^2\\
		&+\sum_{p>\mu}\Big(\alpha(\alpha-1)F^{2\alpha-2}\kappa_p^{-2}+\frac{2\alpha^2F^{2\alpha-3}\kappa_1}{\kappa_p-\kappa_1}(\kappa_1^{-1}-\frac{\alpha-1}{\alpha}\kappa_p^{-1})\frac{\partial F}{\partial h_{pp}}\Big)(\nabla_pF)^2\geq 0.
	\end{aligned}
\end{equation}
Recall that $Q_1\geq 0$, we get that $Q=Q_1+Q_2+Q_3\geq 0$, hence we have proved \eqref{goal2}.
By applying maximum principle to \eqref{LTkl}, $T_{ij}v^iv^j$ can not attain a maximum unless there is a parallel vector field $v=e_1$ such that $T_{ij}v^iv^j\equiv 0$,
which implies that all the nonnegative terms above must be identically zero on $M^n$. Note that $F=\sum\limits_k\frac{\partial F}{\partial h_{kk}}\kappa_k> \frac{\partial F}{\partial h_{11}}\kappa_1$, $\kappa_p>\kappa_1$ when $p>\mu$, from the last inequality of \eqref{Q2Q3}, we obtain that $\nabla_1F=0$ and $\nabla_pF=0$
for $p>\mu$. Due to the same reason with \eqref{info}, we also have that  $\nabla_mF=0$
for $2\leq m\leq\mu$. Hence, we obtain that $F$ is identically constant on $M^n$. Since $M^n$ is strictly convex, it follows immediately from  the equation \eqref{fk} that $M^n$ is a round sphere in Euclidean case and $M^n$ is a slice $\{ r_{0} \}\times \mathbb{S}^{n}$ in hemisphere case.
\end{proof}

%\bibliographystyle{amsplain}
%\bibliography{refers}

\end{document}